\documentclass{amsart}
\numberwithin{equation}{section}

  \usepackage{verbatim}
\usepackage{mathrsfs}
\usepackage{latexsym}
\usepackage{epsfig}
\usepackage{amsmath}
\usepackage{bbm}
\usepackage{graphics}
\usepackage{amssymb}
\usepackage{amsthm}
\usepackage[alphabetic]{amsrefs}
\usepackage{amsopn}
\usepackage{amscd}
\usepackage{tikz}
\usepackage[all,knot]{xy}
\xyoption{all}
\usepackage{rotating}
\usepackage{hyperref}

\theoremstyle{plain}
\newtheorem{thm}{Theorem}[section]
\newtheorem{lem}[thm]{Lemma}
\newtheorem{prop}[thm]{Proposition}
\newtheorem{cor}[thm]{Corollary}

\newtheorem*{thm*}{Theorem}
\newtheorem*{lem*}{Lemma}
\newtheorem*{prop*}{Proposition}
\newtheorem*{cor*}{Corollary}

\theoremstyle{definition}
\newtheorem{defn}[thm]{Definition}
\newtheorem*{defn*}{Definition}

{}
\newtheorem{rem}[thm]{Remark}
\newtheorem*{rem*}{Remark}

\newtheorem{notation}[thm]{Notation}{}
{}
{}
\newtheorem*{ack}{Acknowledgements}{}

\theoremstyle{remark}
{}
{}
{}

\def\cie{\subseteq}

\def\iso{\cong}

\def\to{\longrightarrow}

\def\rimp{\Rightarrow}

\def\S{\Sigma}

\def\Heyt{\mathsf{Heyt}}
\def\cHeyt{\mathsf{cHeyt}}
\def\Bool{\mathsf{Bool}}
\def\cBool{\mathsf{cBool}}

\def\sfD{\mathsf{D}}

\def\sfT{\mathsf{T}}
\def\sfS{\mathsf{S}}

\def\sfX{\mathsf{X}}

\def\op{\mathrm{op}}

\DeclareMathOperator{\Loc}{\mathrm{Loc}}
\DeclareMathOperator{\loc}{\mathrm{loc}}

\title{Complete Boolean algebras are Bousfield lattices}
\author{Greg Stevenson}
\address{Greg Stevenson, School of Mathematics and Statistics,
University of Glasgow,
University Place,
Glasgow G12 8SQ
}
\email{gregory.stevenson@glasgow.ac.uk}
\urladdr{http://www.maths.gla.ac.uk/~gstevenson/}

\begin{document}

\begin{abstract}
\noindent Given a complete Heyting algebra we construct an algebraic tensor triangulated category whose Bousfield lattice is the Booleanization of the given Heyting algebra. As a consequence we deduce that any complete Boolean algebra is the Bousfield lattice of some tensor triangulated category. Using the same ideas we then give two further examples illustrating some interesting behaviour of the Bousfield lattice.
\end{abstract}

\maketitle


\section{Introduction}

A sizable amount of effort has, in the past decades, been invested into understanding lattices of subcategories of triangulated categories. There have been numerous successes, but despite these advances we still know very little in general. The situation improves if one imposes some additional structure: in many cases one runs into tensor triangulated categories, i.e.\ triangulated categories equipped with a compatible symmetric monoidal structure. In this case one can study ideals\textemdash subcategories closed under tensoring. Restricting to the lattice of ideals improves the situation somewhat, but what we know is still very limited.

One can then further restrict to ideals which arise as the kernel of tensoring with some object. Such an ideal is called a Bousfield class, after the pioneering work of Bousfield \cite{BousfieldLSH, BousfieldBoolean}, and the collection of such classes is called the Bousfield lattice; the name is slightly misleading as a priori there could be a proper class of Bousfield classes. More is known about the Bousfield lattice, for instance it was shown by Ohkawa in \cite{Ohkawa} that there are a set of Bousfield classes in the stable homotopy category, so the Bousfield lattice is in fact an honest lattice in that case. This was generalised to arbitrary well generated tensor triangulated categories by Iyengar and Krause \cite{IKBousfield}.

There are many results about Bousfield lattices in various cases, but again what one can hope to prove in general remains rather a mystery. The aim of this note is to give some naively constructed examples of tensor triangulated categories which illustrate certain properties of the Bousfield lattice; it serves to show that one can not be too ambitious in trying to prove general results.

We show in particular that any complete Boolean algebra arises as the Bousfield lattice of a `nice' tensor triangulated category. This is achieved by realizing the Booleanization functor on complete Heyting algebras via triangulated categories. A slight modification of this construction is also used to give an example of a localizing tensor ideal which is not a Bousfield class and to show that the property of having all radical tensor ideals arise as Bousfield classes is not preserved under localization.

\begin{ack}
I'm thankful to the referee for their careful reading of the manuscript; they provided several helpful comments which resulted in improvements to the exposition.
\end{ack}

\section{Recollections on Heyting and Boolean algebras}

In this section we give brief reminders on a number of definitions we will make use of including those of Heyting algebras, Boolean algebras, and the Booleanization functor. Further details, including proofs omitted here, can be found in Section 1 of \cite{StoneSpaces}.

\begin{defn}
A \emph{lattice} $L$ is a poset $(L, \leq)$ such that every pair of elements admits a greatest lower bound (\emph{meet}) denoted $\wedge$ and a least upper bound (\emph{join}) denoted $\vee$. We say $L$ is \emph{bounded} if it admits a minimal and a maximal element, \emph{complete} if arbitrary subsets admit meets and joins, and \emph{distributive} if for all $l_1, l_2, l_3 \in L$ we have
\begin{displaymath}
l_1 \wedge (l_2 \vee l_3) = (l_1 \wedge l_2) \vee (l_1 \wedge l_3).
\end{displaymath}
A lattice $L$ is a \emph{frame} if it is complete and satisfies the stronger distributivity condition that
\begin{displaymath}
l \wedge (\bigvee_{i\in I} m_i) = \bigvee_{i\in I} (l\wedge m_i)
\end{displaymath}
for all $l\in L$ and any family of elements $\{m_i\}$ of $L$ indexed by some set $I$.
\end{defn}

\begin{notation}
All our lattices are assumed to be bounded. Given a lattice $L$ we denote by $0$ and $1$ (or $0_L$ and $1_L$ if our notation would otherwise be ambiguous) the minimal and maximal elements of $L$ respectively.
\end{notation}

\begin{rem}\label{rem_cat1}
In more categorical language a poset is just, up to equivalence, a category in which every hom-set contains at most one element (a \emph{thin category}); the hom-set from an object $l$ to an object $m$ being non-zero precisely if $l\leq m$. From this point of view $L$ is a bounded lattice precisely if the corresponding category is finitely complete and cocomplete. It is distributive if finite products commute with finite colimits.
\end{rem}

\begin{defn}
A \emph{Heyting algebra} is a lattice $H$ equipped with an implication operation $\rimp\colon H^\op \times H \to H$ satisfying the following universal property:
\begin{displaymath}
(l \wedge m) \leq n \;\; \text{if and only if} \;\; l \leq (m \rimp n)
\end{displaymath}
for all $l,m,n \in H$. A \emph{complete Heyting algebra} is a Heyting algebra which is complete as a lattice.

A morphism of Heyting algebras is a morphism of lattices which preserves implication. A morphism of complete Heyting algebras is a morphism of lattices which preserves implication and arbitrary joins.
\end{defn}

If one thinks of lattice elements as `statements', viewing $\wedge$ as `and' and $\vee$ as `or', then the implication operation can really be thought of as implication, i.e.\ $a \rimp b$ should be read as `$a$ implies $b$'. For instance, one can verify the identities
\begin{displaymath}
(l \rimp l) = 1 \quad \text{and} \quad l \wedge (l \rimp m) = l \wedge m
\end{displaymath}
which can be interpreted as `$l$ implies $l$ is true` and `$l$ and $l$ implies $m$ is equivalent to $l$ and $m$'.

\begin{rem}
The implication operation also admits a pleasing categorical description. Following on from Remark~\ref{rem_cat1} a complete Heyting algebra is nothing more than a Cartesian closed thin category. The implication operator is just a right adjoint to the categorical product $\wedge$ i.e.\ it's an internal hom.
\end{rem}

It is standard that Heyting algebras are distributive and that complete Heyting algebras are the same as frames i.e.\ complete distributive lattices. In a frame $L$ the implication operation $l \rimp m$, for $l,m\in L$, is given by the join of the set $\{x\in L \; \vert \; x\wedge l \leq m\}$. However, not every morphism of frames preserves the implication operation. We denote the category of (complete) Heyting algebras and (complete) Heyting algebra morphisms by $\Heyt$ ($\cHeyt$).

\begin{defn}
Let $H$ be a Heyting algebra. The \emph{pseudocomplement} or \emph{negation} of $h\in H$ is defined to be
\begin{displaymath}
\neg h := (h\rimp 0).
\end{displaymath}
The element $\neg h$ is the largest element of $H$ such that $h\wedge \neg h = 0$. We say $h$ is \emph{regular} if $\neg \neg h = h$ and that $h$ is \emph{complemented} if $h \vee \neg h = 1$. We recall that regularity of $h$ is equivalent to the condition that there exists $h'$ with $h = \neg h'$ and that every complemented element of a Heyting algebra is regular.

We denote by $H_{\neg \neg}$ the subposet of regular elements of $H$. There is an induced structure of a Heyting algebra on $H_{\neg \neg}$, but we note that $H_{\neg \neg}$ may not be a Heyting subalgebra of $H$. 
\end{defn}

\begin{rem}
The meet $\wedge$, implication $\rimp$, and the minimal and maximal elements $0$ and $1$ for the induced Heyting algebra structure on $H_{\neg \neg}$ coincide with those of $H$. However, the join can be different. The join of $l,m\in H_{\neg \neg}$ can be expressed in $H$ as
\begin{displaymath}
l \vee_{H_{\neg \neg}} m = \neg(\neg l \wedge \neg m)
\end{displaymath}
and this may not coincide with $l\vee m$.
\end{rem}

\begin{defn}
A (complete) Heyting algebra $B$ is a \emph{(complete) Boolean algebra} if it satisfies one (and hence both) of the following equivalent conditions:
\begin{itemize}
\item[(1)] every element of $B$ is regular;
\item[(2)] every element of $B$ is complemented.
\end{itemize}
A morphism of (complete) Boolean algebras is just a morphism of (complete) Heyting algebras.
\end{defn}

We denote the category of (complete) Boolean algebras and morphisms of (complete) Boolean algebras by $\Bool$ ($\cBool$).

Let $H$ be a Heyting algebra. Then double negation defines a morphism of Heyting algebras $\neg \neg \colon H \to H_{\neg \neg}$. In fact double negation gives a monad on $\Heyt$. If $H$ is complete then so is $H_{\neg \neg}$ and double negation gives a monad on $\cHeyt$. This is summarised in the following well known theorem.

\begin{thm}\label{thm_booleanization}
Let $H$ be a (complete) Heyting algebra. Then the poset $H_{\neg \neg}$ is a (complete) Boolean algebra called the \emph{Booleanization} of $H$. Furthermore, $H \mapsto H_{\neg \neg}$ extends to a functor
\begin{displaymath}
B\colon \Heyt \to \Bool
\end{displaymath}
which is left adjoint to the fully faithful inclusion $\Bool \to \Heyt$. The double negation $H \to H_{\neg \neg}$ is the unit of this adjunction. Moreover, the functor $B$ restricted to $\cHeyt$ gives a left adjoint to the fully faithful inclusion $\cBool \to \cHeyt$.
\end{thm}

\begin{rem}\label{rem_tech}
A particular consequence of the theorem is that for any Heyting algebra $H$ the double negation operation $\neg \neg$ preserves implication; this is part of the statement that the unit $H\to H_{\neg \neg}$ is a morphism of Heyting algebras.
\end{rem}

\section{Any complete Boolean algebra is a Bousfield lattice}

In this section we construct, starting from a complete Heyting algebra $L$, an algebraic tensor triangulated category $\sfT_L$ whose Bousfield lattice is the Booleanization of $L$. Hence if $L$ is a complete Boolean algebra then the Bousfield lattice of $\sfT_L$ is precisely $L$.

We fix a complete Heyting algebra $L$ and a field $k$. Denote by $\sfD(k)$ the unbounded derived category of all $k$-vector spaces. We will denote the suspension in $\sfD(k)$ by $\Sigma$, and also use this to denote suspension in the other triangulated categories which occur.

We define a triangulated category $\sfT_L$ by
\begin{displaymath}
\sfT_L = \prod_{l\in L\setminus \{0_L\}} \sfD(k_l)
\end{displaymath}
where $k_l = k$ for all $l\in L\setminus 0_L$ (the subscript is merely to keep track of indices) and the product is taken in the category of additive categories. Let us unwind what this means and fix some notation. For a $k$-vector space $V$ we use $\S^i V_l$ to denote the object of $\sfT_L$ which is $\S^i V$ in the $l$th position and zero elsewhere. With this convention an object $(X_l)_{l\in L\setminus \{0\}}$ of $\sfT_L$ can be thought of as $\coprod_l X_l$ and we will use this notation. The sense in which one makes this precise is that objects of $\sfT_L$ are the same as $L$-graded objects of $\sfD(k)$ (where $L$ is just viewed as a set) and we can represent such an object either as the sequence of its components, which is the natural description as a product of additive categories, or as the direct sum (in the derived category of vector spaces) of its graded pieces. There is no ambiguity as the constant sequence $(k_l)_{l\in L\setminus \{0\}}$ is both the product and the coproduct of the $k_l$ in $\sfT_L$ (this is a subtle psychological point, but one can convince oneself by thinking about the analogous situation for graded vector spaces).

The triangulated structure is just given levelwise by the usual triangulated structure on the unbounded derived category of $k$-vector spaces. It is easily seen that $\sfT_L$ is compactly generated, pure-semisimple, and has a stable combinatorial model (see \cite{KrTele}*{Section~2.4} for details on pure-semisimplicity and \cite{BarwickModels}*{(1.21.2)} for the definition of combinatorial model category). For convenience we will often denote the zero object $0$ of $\sfT_L$ by $k_{0_L}$ and think of it as the ``generator corresponding to $0_L$''  (thus we do not worry about excluding $0_L$ when indexing generators of $\sfT_L$ over $L$). 


We now use the lattice structure of $L$ to define an exact symmetric monoidal structure on $\sfT_L$. For $l,l' \in L$ we set
\begin{displaymath}
k_l \otimes k_{l'} = k_{l\wedge l'}
\end{displaymath}
where we use the identification $k_{0_L} = 0$. This extends to a monoidal structure on $\sfT_L$ which is exact and coproduct preserving in each variable and with unit object $k_1$. Indeed, every object of $\sfT_L$ is a sum of suspensions of the $k_l$ for $l\in L$ so the rule above, together with the usual tensor product on $\sfD(k)$, determines an (essentially unique) exact coproduct preserving extension.

\begin{rem}\label{rem_monoidal}
In fact we did not use the full lattice structure in order to concoct $\sfT_L$, we only really needed the monoid $(L, \wedge, 1)$. Given any monoid $M$ we can define a tensor triangulated category $\sfT_M$ exactly as above, with all of the same good properties. We shall exploit this, for some simple minded monoids, in Sections \ref{sec:not} and \ref{sec:poor}.
\end{rem}

We now define the object which will occupy our interest for the remainder of this note.

\begin{defn}
Let $\sfT$ be a compactly generated tensor triangulated category. Given an object $X\in \sfT$ its \emph{Bousfield class} is the full subcategory
\begin{displaymath}
A(X) = \{Y\in \sfT \; \vert \; X\otimes Y = 0\}.
\end{displaymath}
We note that $A(X)$ is triangulated, closed under tensoring with arbitrary objects, and closed under retracts and arbitrary coproducts, i.e.\ it is a \emph{localizing tensor ideal}.

We denote by $A(\sfT)$ the collection of all Bousfield classes of $\sfT$ and call it the \emph{Bousfield lattice} of $\sfT$. This is naturally a poset under reverse inclusion and, in fact, forms a complete lattice: the join of the Bousfield classes $\{A(X_i)\; \vert \; i\in I\}$ is given by $A(\coprod_i X_i)$.
\end{defn}

\begin{notation}
Throughout we use $\wedge$ and $\vee$ to denote the meet and join of the Bousfield lattice; this overlaps with our notation internal to $L$, but it should always be clear from the context which usage is being invoked.
\end{notation}

We will also need a small amount of notation for dealing with triangulated categories.

\begin{notation}
Let $\sfT$ be a compactly generated triangulated category and let $\sfX$ be a class of objects of $\sfT$. We denote by $\loc(X)$ the smallest localizing subcategory containing $\sfX$, i.e.\ the smallest subcategory of $\sfT$ which contains $\sfX$ and is closed under suspensions, cones, and arbitrary coproducts.
\end{notation}

Let us return to studying $\sfT_L$. We recall from \cite{IKBousfield} that $A(\sfT_L)$ is a set rather than a proper class (this will become clear through explicit computation in any case). Thus there are no set theoretic issues with any construction using Bousfield classes.

Note that, for any object $X$ of $\sfT_L$, the class $A(X)$ is determined by the $k_l$ it contains; this is immediate as Bousfield classes are closed under summands and every object of $\sfT_L$ is a sum of suspensions of the $k_l$.

The first real observation we make is that every Bousfield class in $\sfT_L$ comes from one of the standard generators $k_l$.

\begin{lem}
Let $X$ be an object of $\sfT_L$. Then there exists an $l\in L$ such that 
\begin{displaymath}
A(X) = A(k_l).
\end{displaymath}
\end{lem}
\begin{proof}
We can write $X$ as a coproduct of suspensions of the $k_m$ for $m\in L$ and, since suspension does not change Bousfield classes, we may without loss of generality take $X\ \iso \coprod_i k_{l_i}$ where $l_i\in L$ are lattice elements indexed by some set $I$. Let us write $l$ for the join $\vee_i l_i$ of the $l_i$. Then
\begin{align*}
A(X) &= A(\coprod_i k_{l_i}) \\
&= \loc( k_m \; \vert \;  (\coprod_i k_{l_i}) \otimes k_m = 0) \\
&= \loc( k_m \; \vert \;  \coprod_i (k_{l_i} \otimes k_m) = 0 ) \\
&= \loc( k_m \; \vert \;  \coprod_i (k_{l_i\wedge m}) = 0 ) \\
&= \loc( k_m \; \vert \;  l_i \wedge m = 0_L \;\; \forall i\in I ) \\
&= \loc( k_m \; \vert \; l \wedge m = 0_L ) \\
&= A(k_l).
\end{align*}
\end{proof}

\begin{rem}\label{rem_join}
We proved a little more than we stated in the Lemma and we wish to record it for use later. Namely, we showed that if $\{l_i\}_{i\in I}$ is a set of elements of $L$ then $A(\coprod_i k_{l_i}) = A(k_{\vee_i l_i})$.
\end{rem}

Next we shall describe the meet on $A(\sfT_L)$. By the last lemma it is sufficient to do this for the Bousfield classes of the $k_l$ with $l\in L$. Given objects $X$ and $Y$ of $\sfT_L$ we set 
\begin{displaymath}
A(X)\otimes A(Y) := A(X\otimes Y).
\end{displaymath}
Observe that if $A(X)\leq A(X')$ and $A(Y)\leq A(Y')$ then there is an inequality $A(X) \otimes A(Y) \leq A(X') \otimes A(Y')$.

\begin{lem}\label{lem_meet}
Let $l$ and $l'$ be elements of $L$. Then the meet of $A(k_{l})$ and $A(k_{l'})$ is given by $A(k_l)\otimes A(k_{l'}) = A(k_{l\wedge l'})$.
\end{lem}
\begin{proof}
We need to show that $A(k_{l\wedge l'})$ is the greatest lower bound for $A(k_l)$ and $A(k_{l'})$. So suppose $A(k_x) \leq A(k_l)$ and $A(k_x) \leq A(k_{l'})$. Then
\begin{displaymath}
A(k_x) = A(k_{x\wedge x}) = A(k_x)\otimes A(k_x) \leq A(k_l) \otimes A(k_{l'}) = A(k_{l\wedge l'})
\end{displaymath}
which shows that $A(k_{l\wedge l'})$ is the greatest lower bound for $A(k_l)$ and $A(k_{l'})$.
\end{proof}

It follows easily that $A(\sfT_L)$ is a frame (aka a complete Heyting algebra).

\begin{lem}\label{lem_frame}
The Bousfield lattice $A(\sfT_L)$ is a frame, i.e.\ finite meets distribute over infinite joins.
\end{lem}
\begin{proof}
Let $l$ be an element of $L$ and $\{m_i\}_{i\in I}$ a set of elements of $L$ indexed by a set~$I$. Then
\begin{align*}
A(k_l) \wedge (\bigvee_i A(k_{m_i})) & = A(k_l) \otimes A(\coprod_i k_{m_i}) \\
&= A(k_l \otimes (\coprod_i k_{m_i})) \\
&= A(\coprod_i (k_l\otimes k_{m_i})) \\
&= \bigvee_i A(k_l \otimes k_{m_i}) \\
&= \bigvee_i (A(k_l) \wedge A(k_{m_i}))
\end{align*}
where we have used Lemma \ref{lem_meet} for the first and last equalities. 
\end{proof}

We can then define the implication operation for the classes corresponding to $l,m\in L$ by
\begin{displaymath}
(A(k_l) \rimp A(k_{m})) = \bigvee_{A(k_x)\wedge A(k_l) \leq A(k_m)} A(k_x).
\end{displaymath}
In fact $A(\sfT_L)$ is a complete Boolean algebra. Before proving this let us give a concrete description of the negation operation on the Bousfield lattice. This negation operator is a general form of the one originally considered by Bousfield \cite{BousfieldLSH} (also see \cite{HPBousfield}).

\begin{lem}\label{lem_negation}
For $l\in L$ there is an equality
\begin{displaymath}
\neg A(k_l) = A(\coprod_{k_m \in A(k_l)} k_m) = A(k_{\neg l}).
\end{displaymath}
\end{lem}
\begin{proof}
Recall that $\neg A(k_l)$ is, by definition, $A(k_l) \rimp A(0)$ where $A(0) = \sfT$. Explicitly we have
\begin{displaymath}
\neg A(k_l) = \bigvee_{A(k_x)\wedge A(k_l) \leq A(0)} A(k_x).
\end{displaymath}
Now the lattice elements representing Bousfield classes in the indexing set in the wedge occurring above can be rewritten, using Lemma \ref{lem_meet}, as
\begin{displaymath}
\{x\in L \; \vert \; A(k_{x\wedge l}) \leq A(0) \} = \{x\in L \; \vert \; k_{x\wedge l} = 0\} = \{x\in L \; \vert \; x\wedge l = 0_L\}
\end{displaymath}
which is precisely the indexing set of the join occurring in the explicit definition of $(l\rimp 0_L)$. Together with Remark \ref{rem_join} this gives the claimed equalities.
\end{proof}

\begin{prop}\label{prop_boolean}
The Bousfield lattice of $\sfT_L$ is a complete Boolean algebra.
\end{prop}
\begin{proof}
We have just seen that $A(\sfT_L)$ is a frame. We will show that every element of $A(\sfT_L)$ is regular i.e.\ check that $\neg \neg A(k_l) = A(k_l)$ for all $l\in L$.

Let $l$ be an element of $L$. Since $l\wedge \neg l = 0_L$ we have $k_l \otimes k_{\neg l} = 0$. In particular, $k_l$ is a summand in $\coprod_{k_m \in A(k_{\neg l})} k_m$ which, by the last lemma gives an object representing $\neg \neg A(k_l)$. Thus $\neg \neg A(k_l) \geq A(k_l)$.

On the other hand suppose $k_w \otimes k_l = 0$ so $k_w\in A(k_l)$, i.e.\ $w\wedge l = 0_L$. We know, by the last lemma, that $\neg A(k_l) \geq A(k_w)$. The double negation is given by
\begin{displaymath}
\neg(\neg A(k_l)) = A(\coprod_{k_x \in \neg A(k_l)} k_x)
\end{displaymath}
where, by definition, each $k_x$ lies in $\neg A(k_l)$ and hence in $A(k_w)$. Thus $k_w \in \neg \neg A(k_l)$ showing $A(k_l) \geq \neg \neg A(k_l)$ and completing the proof.
\end{proof}

We now want to compare the lattice $L$ to $A(\sfT_L)$. Define an assignment $\phi\colon L \to A(\sfT_L)$ by $\phi(l) = A(k_l)$. Our claim is that $\phi$ is a well defined morphism of complete Heyting algebras which identifies $A(\sfT_L)$ with the Booleanization of $L$. It is clear that $\phi$ is well defined and it is surjective as we have noted above that every Bousfield class is of the form $A(k_l)$.

\begin{lem}
The assignment $\phi\colon L \to A(\sfT_L)$ is a morphism of frames.
\end{lem}
\begin{proof}
We need to check that $\phi$ is monotone, and preserves finite meets and infinite joins.

Let us first show that $\phi$ is monotone. Suppose $l \leq m$ in $L$. Then for all $x\in L$ we have $x\wedge l \leq x\wedge m$ and so $x\wedge m = 0_L$ implies that $x\wedge l = 0_L$. Hence $A(k_l) \leq A(k_m)$ as required.

The map $\phi$ preserves finite meets by Lemma \ref{lem_meet}. In order to show it preserves joins let $\{l_i\}_{i\in I}$ be a set of elements of $L$ with join $l$. We have
\begin{displaymath}
\phi(\bigvee_i l_i) = A(k_{\vee_i l_i}) = A(\coprod_i k_{l_i}) = \bigvee_i A(k_{l_i}) = \bigvee_i \phi(l_i),
\end{displaymath}
where the second equality is Remark \ref{rem_join}.
\end{proof}

The next lemma shows that $\phi$ is in fact a morphism of complete Heyting algebras.

\begin{lem}
The map $\phi$ preserves implication.
\end{lem}
\begin{proof}
To start with let us recall that by Lemma \ref{lem_negation} the map $\phi$ commutes with negation, and by Proposition \ref{prop_boolean} the lattice $A(\sfT_L)$ is a Boolean algebra and so negation is an involution on $A(\sfT_L)$. Thus for all $l\in L$ we have $\phi(\neg \neg l) = \phi(l)$.

We need to check that
\begin{align*}
\phi(l\rimp m) = \phi(\neg \neg(l\rimp m)) = \phi((\neg \neg l) \rimp (\neg \neg m)) &= \phi(\bigvee_{x\wedge (\neg \neg l) \leq (\neg \neg m)} x) \\
&= \bigvee_{x\wedge (\neg \neg l) \leq (\neg \neg m)} A(k_x)
\end{align*}
agrees with
\begin{displaymath}
\phi(l)\rimp \phi(m) = \bigvee_{A(k_x) \wedge A(k_l) \leq A(k_m)} A(k_x).
\end{displaymath}
Here, in order to rewrite $\phi(l\rimp m)$ in the first string of equalities, we have used the observation of Remark~\ref{rem_tech} that double negation preserves implication (and that $L_{\neg \neg}$ is closed under implication in $L$).

One direction is straight forward. Since $\phi$ is order preserving we have $\phi(l\rimp m) \leq (\phi(l)\rimp \phi(m))$; if $x\wedge l \leq m$ then $A(k_x)\wedge A(k_l) \leq A(k_m)$.

On the other hand suppose $A(k_y)\wedge A(k_l) \leq A(k_m)$ i.e.\ for $z\in L$ we have $m\wedge z = 0_L$ implies that $(y\wedge l)\wedge z = 0_L$. In other words there is a containment
\begin{displaymath}
\{z\in L \; \vert \; m\wedge z \leq 0_L\} \cie \{z\in L \; \vert \; (y\wedge l)\wedge z \leq 0_L\}
\end{displaymath}
of the index sets defining the negations of $m$ and $y\wedge l$. Thus $\neg m \leq \neg (y\wedge l)$ and so $\neg \neg m \geq \neg \neg (y\wedge l) = (\neg \neg y)\wedge (\neg \neg l)$.  Hence $A(k_{\neg \neg y}) = A(k_y)$ occurs in the wedge defining $\phi(\neg \neg (l\rimp m))$ and we see that $(\phi(l)\rimp \phi(m)) \leq \phi(\neg \neg(l\rimp m)) = \phi(l\rimp m)$ completing the proof.
\end{proof}

We are now ready to show that our construction gives a (somewhat arcane) realization of the Booleanization of $L$.

\begin{prop}
The morphism $\phi$ induces an isomorphism of complete Boolean algebras $A(\sfT_L) \iso L_{\neg \neg}$.
\end{prop}
\begin{proof}
By the universal property of the Booleanization (see Theorem \ref{thm_booleanization}) the morphism $\phi$ must factor via a unique map of complete Boolean algebras $\phi' \colon L_{\neg \neg} \to A(\sfT_L)$. Since $\phi$ is surjective so is $\phi'$.

The map $\phi'$ is also injective. Indeed, $\phi'(\neg \neg l) = \phi'(\neg \neg m)$ if and only if $A(k_l) = A(k_m)$ if and only if
\begin{displaymath}
\{z\in L_{\neg \neg} \;  \vert \;  z\wedge (\neg \neg l) = 0_L\} = \{z\in L_{\neg \neg} \; \vert \; z\wedge (\neg \neg m) = 0_L\},
\end{displaymath}
which can happen if and only if $\neg \neg l = \neg \neg m$. This last statement is a consequence of the fact that elements of a Boolean algebra are completely determined by their annihilators \cite{Vechtomov}.
\end{proof}

In particular, if we started with a Boolean algebra $L$ we would recover it as $A(\sfT_L)$.

\begin{cor}
Every complete Boolean algebra occurs as the Bousfield lattice of an algebraic tensor triangulated category which can be presented as the homotopy category of a combinatorial stable monoidal model category.
\end{cor}

This gives a significant class of lattices that occur as Bousfield lattices (albeit in this instance via a somewhat artificial construction). It is known that the Bousfield lattice of a triangulated category need not be Boolean, nor even distributive. It would be very informative to have a more intricate construction that allowed one to realise a prescribed complete lattice as a Bousfield lattice or to characterise the lattices which can occur.

\section{Not every localizing ideal is a Bousfield class}\label{sec:not}
We give an example of a rather simple tensor triangulated category possessing a localizing $\otimes$-ideal which is not a Bousfield class; this shows that Conjecture~9.1 of \cite{HPBousfield} need not be true in an arbitrary tensor triangulated category. Another such example, which is more natural at the cost of greater technicality, is given in \cite{StevensonAbsFlat}.

One can view our example here as a special case of the construction involving lattices. However, we instead choose to describe it via a similar construction using a monoid as in Remark~\ref{rem_monoidal}.

Let $k$ be a field and let $M$ denote the monoid with two elements $\{1,m\}$ where $1$ is the identity and $m^2 = m$. Let $\sfT_M$ denote the triangulated category constructed as in Remark~\ref{rem_monoidal} i.e.\ $\sfD(k_1) \oplus \sfD(k_m)$, where $k_1 = k_m = k$, and we define a tensor product $\otimes$ on $\sfT_M$ using the monoid structure of $M$: $k_1$ is the unit object and $k_m \otimes k_m \iso k_m$. 

It is clear that $A(\sfT_M)$ is the two element lattice consisting of $A(0)$ and $A(k_1)$. On the other hand $\loc(k_m)$ is a non-zero proper tensor ideal of $\sfT_M$ and one sees there is no object $X$ of $\sfT_M$ such that $\loc(k_m) = A(X)$.

The category $\sfT_M$ also gives an example of a smashing localization where the acyclization functor is given by tensoring with an idempotent object but the localization functor is not. In particular, $\sfT_M$ cannot be rigidly compactly generated i.e.\ the full subcategory of compact objects is not rigid (see \cite{Stevenson16pp}*{Definition~1.3} for a discussion of rigidity). To see this note that $\loc(k_m)$ is, as noted above, a localizing $\otimes$-ideal generated by the compact object $k_m$. It is thus smashing and one sees easily that the corresponding acyclization functor is given by $k_m \otimes (-)$. However, there is no tensor idempotent realizing the localization functor, given by projection onto the component $D(k_1)$, as such an idempotent would have to be tensor orthogonal to $k_m$ and no such non-zero object exists in $\sfT_M$.

\section{Poor behaviour of the Bousfield lattice under localization}\label{sec:poor}

We now give a slightly modified version of the last example which shows that the Bousfield lattice is not necessarily well behaved under localization by a tensor ideal; in particular we show that localizing by the ``nilradical'' can destroy the good behaviour of the Bousfield lattice. To begin, let us define what we mean by the nilradical of a tensor triangulated category.

\begin{defn}
Let $\sfT$ be a tensor triangulated category with small coproducts. The \emph{nilradical} of $\sfT$, denoted $\sqrt{\sfT}$, is the smallest radical localizing $\otimes$-ideal containing $0$. Explicitly, it is the smallest localizing subcategory of $\sfT$ which is closed under tensoring with all objects of $\sfT$ and has the property that if it contains $Y^{\otimes n}$ then it contains $Y$.
\end{defn}

We will use the same construction as in Remark~\ref{rem_monoidal} and the last section. Let $M$ be the commutative monoid $\{0,1,x,y\}$ with multiplication table
\begin{displaymath}
\begin{tabular}{| c || c | c| c| c|}
\hline
  & 0 & 1 & x & y \\ \hline
  &   &   &   &   \\[-1.0em] \hline
0 & 0 & 0 & 0 & 0 \\ \hline 
1 & 0 & 1 & x & y \\ \hline
x & 0 & x & 0 & 0 \\ \hline
y & 0 & y & 0 & y \\ \hline
\end{tabular}
\end{displaymath}
We associate to $M$ a tensor triangulated category 
\begin{displaymath}
\sfT_M = \sfD(k_1)\oplus \sfD(k_x) \oplus \sfD(k_y)
\end{displaymath}
with tensor product defined using the multiplication of $M$. One checks easily that the spectrum of localizing prime tensor ideals of $\sfT_M$ (which agrees with the spectrum, in the sense of \cite{BaSpec}, of the compacts) has two points $\langle k_x \rangle = \sqrt{\sfT}$ and $\langle k_x, k_y \rangle$. It is the same as the topological space underlying the spectrum of a discrete valuation ring. Both of these ideals are Bousfield classes:
\begin{displaymath}
\langle k_x \rangle = A(k_y) \quad \text{and} \quad \langle k_x, k_y \rangle = A(k_x).
\end{displaymath}
Denoting by $\Loc^{\sqrt{\otimes}}(\sfT_M)$ the collection of radical localizing $\otimes$-ideals we have
\begin{displaymath}
A(\sfT_M) = \Loc^{\sqrt{\otimes}}(\sfT_M) \cup \{\langle 0 \rangle\} = \{\sfT_M, \langle k_x \rangle, \langle k_x, k_y\rangle, \langle 0 \rangle\}.
\end{displaymath}
There is also a non-radical ideal, namely $\langle k_y \rangle$, which is not a Bousfield class.

It is easily seen that forming the quotient $\sfS = \sfT_M/\sqrt{\sfT_M}$ does not change the spectrum and gives a bijection $\Loc^{\sqrt{\otimes}}(\sfT_M) \iso \Loc^{\sqrt{\otimes}}(\sfS)$. However, it is no longer true that every radical ideal is a Bousfield class in $\sfS$. Indeed, $\sfS$ is none other than the example from the last section and so the radical ideal $\loc(k_y)$ is not a Bousfield class.

\bibliography{greg_bib}

\begin{bibdiv}
\begin{biblist}

\bib{BaSpec}{article}{
      author={Balmer, P.},
       title={The spectrum of prime ideals in tensor triangulated categories},
        date={2005},
     journal={J. Reine Angew. Math.},
      volume={588},
       pages={149\ndash 168},
}

\bib{BarwickModels}{article}{
      author={Barwick, Clark},
       title={On left and right model categories and left and right {B}ousfield
  localizations},
        date={2010},
     journal={Homology Homotopy Appl.},
      volume={12},
      number={2},
       pages={245\ndash 320},
}

\bib{BousfieldBoolean}{article}{
      author={Bousfield, A.~K.},
       title={The {B}oolean algebra of spectra},
        date={1979},
     journal={Comment. Math. Helv.},
      volume={54},
      number={3},
       pages={368\ndash 377},
}

\bib{BousfieldLSH}{article}{
      author={Bousfield, A.~K.},
       title={The localization of spectra with respect to homology},
        date={1979},
     journal={Topology},
      volume={18},
      number={4},
       pages={257\ndash 281},
}

\bib{HPBousfield}{incollection}{
      author={Hovey, Mark},
      author={Palmieri, John~H.},
       title={The structure of the {B}ousfield lattice},
        date={1999},
   booktitle={Homotopy invariant algebraic structures ({B}altimore, {MD},
  1998)},
      series={Contemp. Math.},
      volume={239},
   publisher={Amer. Math. Soc.},
     address={Providence, RI},
       pages={175\ndash 196},
}

\bib{IKBousfield}{article}{
      author={Iyengar, Srikanth~B.},
      author={Krause, Henning},
       title={The {B}ousfield lattice of a triangulated category and
  stratification},
        date={2013},
     journal={Math. Z.},
      volume={273},
      number={3-4},
       pages={1215\ndash 1241},
}

\bib{StoneSpaces}{book}{
      author={Johnstone, Peter~T.},
       title={Stone spaces},
      series={Cambridge Studies in Advanced Mathematics},
   publisher={Cambridge University Press, Cambridge},
        date={1982},
      volume={3},
}

\bib{KrTele}{article}{
      author={Krause, H.},
       title={Smashing subcategories and the telescope conjecture---an
  algebraic approach},
        date={2000},
     journal={Invent. Math.},
      volume={139},
      number={1},
       pages={99\ndash 133},
}

\bib{Ohkawa}{article}{
      author={Ohkawa, Tetsusuke},
       title={The injective hull of homotopy types with respect to generalized
  homology functors},
        date={1989},
     journal={Hiroshima Math. J.},
      volume={19},
      number={3},
       pages={631\ndash 639},
}

\bib{StevensonAbsFlat}{article}{
      author={Stevenson, Greg},
       title={Derived categories of absolutely flat rings},
        date={2014},
     journal={Homology Homotopy Appl.},
      volume={16},
      number={2},
       pages={45\ndash 64},
}

\bib{Stevenson16pp}{unpublished}{
      author={Stevenson, Greg},
       title={A tour of support theory for triangulated categories through
  tensor triangular geometry},
        date={2016},
        note={Preprint available at \url{http://arxiv.org/abs/1601.03595}},
}

\bib{Vechtomov}{article}{
      author={Vechtomov, E.~M.},
       title={Annihilator characterizations of {B}oolean rings and {B}oolean
  lattices},
        date={1993},
     journal={Mat. Zametki},
      volume={53},
      number={2},
       pages={15\ndash 24},
}

\end{biblist}
\end{bibdiv}
\end{document}